\documentclass[11pt,a4paper]{amsart}
\title{Flatness and Shipley's algebraicization theorem}
\date{}
\usepackage{amsmath,amssymb,amsthm,tikz, dsfont, stmaryrd, float,tikz-cd,mathrsfs, hyperref, mathtools, todonotes, xypic, enumerate, tipa, fullpage}
\usepackage[T1]{fontenc}
\usepackage{textcomp}
\usepackage{lmodern}
\usepackage[mathscr]{euscript}
\tikzset{
	labl/.style={anchor=south, rotate=270, inner sep=.5mm}
}
\newtheorem{thm}{Theorem}[section]

\newtheorem{prop}[thm]{Proposition}
\newtheorem{lem}[thm]{Lemma}
\newtheorem{cor}[thm]{Corollary}

\theoremstyle{definition}
\newtheorem{defn}[thm]{Definition}

\newtheorem{rem}[thm]{Remark}

\newtheorem{eg}[thm]{Example}

\newtheorem{notation}[thm]{Notation}
\newtheorem{hyp}[thm]{Hypothesis}

\newcommand{\mc}{\mathcal}

\renewcommand{\lim}[1]{\mathop{\text{lim}}_{#1}}

\newcommand{\pushout}{\ulcorner}
\newcommand{\map}{\mathrm{map}}

\renewcommand{\phi}{\varphi}
\renewcommand{\epsilon}{\varepsilon}

\newcommand{\res}[1][\theta]{{#1}^*}
\newcommand{\ext}[1][\theta]{{#1}_*}

\renewcommand{\mod}[1]{\mathrm{Mod}_{#1}}

\newcommand{\symsp}{\mathrm{Sp}^\Sigma}

\newcommand{\sset}{\mathrm{sSet}}

\newcommand{\sq}{\mathrm{s}\mathbb{Q}\text{-}\mathrm{mod}}

\newcommand{\chq}{\mathrm{Ch}_\mathbb{Q}^+}

\newcommand{\Sp}{\mathrm{Sp}}

\newcommand{\C}{\mathscr{C}}
\newcommand{\D}{\mathscr{D}}
\begin{document}
\author{Jordan Williamson}
\address{School of Mathematics and Statistics, Hicks Building, Sheffield S3 7RH, UK}
\email{jwilliamson3@sheffield.ac.uk}
\begin{abstract}
We provide an enhancement of Shipley's algebraicization theorem which behaves better in the context of commutative algebras. This involves defining flat model structures as in Shipley and Pavlov-Scholbach, and showing that the functors still provide Quillen equivalences in this refined context. The use of flat model structures allows one to identify the algebraic counterparts of change of groups functors, as demonstrated in forthcoming work of the author. 
\end{abstract}

\maketitle

\setcounter{tocdepth}{1}

\tableofcontents

\section{Introduction}
Many concepts and constructions in algebra can be understood in a homotopy invariant sense, and the derived category of a ring is the universal category in which to study these. In turn, these homotopy invariant algebraic notions can be translated into stable homotopy theory~\cite{ElmendorfKrizMandellMay97} and this translation to \emph{spectral algebra} has led to a powerful new point of view on many areas such as modular representation theory~\cite{DwyerGreenleesIyengar06, Greenlees18c}.  Robinson~\cite{Robinson87} showed that the category of spectra contains `extraordinary' derived categories generalizing the derived category of a ring. Shipley~\cite{Shipley07} gave a more precise and general version of Robinson's result in terms of a zig-zag of Quillen equivalences. This paper is a contribution to the understanding of the relationship between spectral and homological algebra.

Passing between the worlds of spectral algebra and homological algebra is a valuable technique. It allows the reduction of topological questions to algebraic questions, and conversely, allows the importation of algebraic methods to the realm of spectra. Associated to any ring $R$ there is an Eilenberg-MacLane spectrum $HR$, and the homological algebra of $R$ is equivalent to the spectral algebra of $HR$. This relation is particularly striking in the case that $R=\mathbb{Q}$, as the rational sphere spectrum is equivalent to $H\mathbb{Q}$. 

Let $R$ be a commutative ring. It was shown by Shipley~\cite{Shipley07} that there is a zig-zag of Quillen equivalences between $HR$-module spectra and chain complexes of $R$-modules. Moreover, this is a zig-zag of symmetric monoidal Quillen equivalences, so that it gives a zig-zag of Quillen equivalences between $HR$-algebra spectra and differential graded $R$-algebras. 
Shipley's algebraicization theorem shows that spectral algebra is a vast generalization of homological algebra. Moreover, it provides a bridge between the worlds of topology and algebra. This bridge has been widely used in the construction of algebraic models for rational equivariant cohomology theories by Barnes, Greenlees, K\k{e}dziorek and Shipley, see~\cite{Barnes09,Barnes17,BarnesGreenleesKedziorek18,GreenleesShipley11, GreenleesShipley14,GreenleesShipley18,Kedziorek17,PolWilliamson,Shipley02}. 

By Shipley's algebraicization theorem, an $HR$-algebra $X$ corresponds to a differential graded $R$-algebra $\Theta X$ and there is a Quillen equivalence $\mod{X} \simeq_Q \mod{\Theta X}.$ However, if $X$ is in addition a commutative $HR$-algebra, it does not correspond to a commutative differential graded $R$-algebra, but rather to a differential graded $E_\infty$-$R$-algebra, see~\cite{RichterShipley17}.

When $R=\mathbb{Q}$, more is true. A commutative $H\mathbb{Q}$-algebra $X$ does correspond to a commutative differential graded $\mathbb{Q}$-algebra by~\cite[1.2]{Shipley07}. More precisely, there is zig-zag of natural weak equivalences $\Theta X \simeq \Theta' X$ where $\Theta' X$ is a commutative DGA. However, despite the fact that the categories of modules have symmetric monoidal structures, the Quillen equivalence $\mod{X} \simeq_Q \mod{\Theta' X}$ is \emph{not} a symmetric monoidal Quillen equivalence. This is because the upgrading of a Quillen equivalence to the categories of modules involves cofibrant replacement of monoids~\cite[3.12(1)]{SchwedeShipley03} which will destroy commutativity and hence the symmetric monoidal structure.

The stable model structure on spectra does not behave well with respect to commutative algebras, in the sense that for a commutative ring spectrum $S$, cofibrant commutative $S$-algebras are not cofibrant as $S$-modules in general. Shipley~\cite{Shipley04} constructed the flat model structure (also called the $S$-model structure) on symmetric spectra, which does satisfy the property that cofibrant commutative algebras are cofibrant as modules. Pavlov-Scholbach~\cite{PavlovScholbach18} extended this to the case of symmetric spectra in general model categories. This extra compatibility between commutative algebras and modules provides several useful tools that would otherwise not be valid. For a concrete example of where this compatibility can be useful, see the next section of the introduction.

In light of these considerations, the goal of this paper is threefold. Firstly, we show that the zig-zag of Quillen equivalences in Shipley's algebraicization theorem still holds in flat model structures which satisfy the extra compatibility between commutative algebras and modules discussed above. See Section~\ref{sec:shipley1} for a precise statement of the zig-zag of symmetric monoidal Quillen equivalences in this first theorem. 
\begin{thm}\label{thm:main1}
	There is a zig-zag of symmetric monoidal Quillen equivalences \[\mod{H\mathbb{Q}}^\mathrm{flat} \simeq_Q \mathrm{Ch}_\mathbb{Q}\]
	where the intermediate categories have the flat model structure.
\end{thm}
In fact we show that the flat model structures on the intermediate categories are the same as the stable model structures used by Shipley~\cite{Shipley07}, see Corollary~\ref{cor:stable and flat}. 

Secondly, we use this theorem to give a new proof of the following theorem, which appears in the body of the paper as Theorem~\ref{thm:commHQalg}. In particular, our approach does not pass through the category of $E_\infty$-algebras as in the proof given by Richter-Shipley~\cite{RichterShipley17}. 
\begin{thm}
There is a zig-zag of Quillen equivalences between the category of commutative $H\mathbb{Q}$-algebras and the category of commutative rational DGAs.
\end{thm}

Finally, we prove the following theorem which appears as Theorem~\ref{thm:symmmon} in the main body of the paper.
\begin{thm}
For a commutative $H\mathbb{Q}$-algebra $X$ there is a zig-zag of weak symmetric monoidal Quillen equivalences $\mod{X} \simeq_Q \mod{\underline{\Theta}X}$ where $\underline{\Theta}X$ is a commutative DGA.
\end{thm}

\subsection*{Motivation and related work}
The author's main motivation comes from the study of algebraic models for rational equivariant cohomology theories. A key step in the construction of algebraic models is the passage from modules over a commutative $H\mathbb{Q}$-algebra to modules over a commutative DGA via Shipley's algebraicization theorem. Therefore, a deep understanding of Shipley's algebraicization theorem provides key insights into the understanding of algebraic models for rational equivariant cohomology theories.

Working in the flat model structure provides valuable techniques which are not valid in the stable model structure. 
In forthcoming work~\cite{changeofgroups}, the author considers the correspondence of the change of groups functors in rational equivariant stable homotopy theory with functors between the algebraic models. In particular, this includes studying how the extension-restriction-coextension of scalars adjoint triple along a map of commutative $H\mathbb{Q}$-algebras $\theta\colon S \to R$ behaves with respect to the Quillen equivalences in Shipley's algebraicization theorem.

The restriction of scalars functor along a map of commutative monoids $\theta\colon S \to R$ in a symmetric monoidal model category is always right Quillen in the model structure right lifted from the underlying category, but it is not left Quillen in general. If the monoidal unit of the underlying category is cofibrant, then restriction of scalars is left Quillen if and only if $R$ is cofibrant as an $S$-module. Since a key step in the proof of algebraic models is a formality argument based on the fact that polynomial rings are formal as commutative DGAs, one needs to be able to replace $R$ in such a way that it is still a commutative $S$-algebra, and is cofibrant as an $S$-module. This replacement is possible in the flat model structure, but not in the stable model structure on spectra. Therefore, Theorem~\ref{thm:main1} provides the necessary setup in which to attack the correspondence of functors along the bridge which Shipley's algebraicization theorem provides between topology and algebra.

The use of the flat model structure allows the extension of the result to commutative algebra objects, so that we prove a Quillen equivalence between the category of commutative $H\mathbb{Q}$-algebras and the category of commutative rational DGAs. Richter-Shipley~\cite{RichterShipley17} prove that the category of commutative $HR$-algebras is Quillen equivalent to the category of differential graded $E_\infty$-$R$-algebras for any commutative ring $R$. Since $E_\infty$-algebras in chain complexes of $\mathbb{Q}$-modules can be rectified to strictly commutative objects, see for example~\cite[\S 7.1.4]{HA}, as a corollary~\cite[8.4]{RichterShipley17} of Richter and Shipley's result one obtains that the category of commutative $H\mathbb{Q}$-algebras is Quillen equivalent to the category of commutative rational DGAs. We give a concrete zig-zag of Quillen equivalences which lands naturally in commutative DGAs, bypassing the need for the rectification step. We expect that this direct approach will enable a better understanding of algebraic models for naive-commutative rational $G$-spectra as studied by Barnes-Greenlees-K\k{e}dziorek~\cite{BarnesGreenleesKedzioreknaive,BarnesGreenleesKedziorek19}. White-Yau~\cite{WhiteYau19} give an alternative approach to this zig-zag of Quillen equivalences by using the stable model structure and their theory of lifting Quillen equivalences to categories of coloured operads. The generality of their theory leads to more stringent hypotheses than our approach, see for example~\cite[3.27]{WhiteYau19}. Our approach exploits the fact that in the flat model structure, cofibrant commutative algebras forget to cofibrant modules.

Finally we give a concrete zig-zag of \emph{symmetric monoidal} Quillen equivalences between the category of modules over a commutative $H\mathbb{Q}$-algebra and the category of modules over a commutative DGA. The result is assumed without proof in the literature, see for example~\cite[3.4.4]{BarnesGreenleesKedziorekShipley17}. Due to the importance of this result in the construction of algebraic models, we believe it is valuable to make the proof explicit. Shipley proved that there is a Quillen equivalence~\cite[2.15]{Shipley07} between modules over an $HR$-algebra $X$ and modules over a DGA $\Theta X$ for any ring $R$. In the case that $R = \mathbb{Q}$, Shipley furthermore proves that $\Theta X$ is naturally weakly equivalent to a \emph{commutative} DGA $\Theta'X$~\cite[1.2]{Shipley07}. A dual of a result of Schwede-Shipley~\cite[3.12(2)]{SchwedeShipley03} allows one to conclude moreover that there is a commutative DGA $\underline{\Theta}X$ and a zig-zag of \emph{symmetric monoidal} Quillen equivalences $\mod{X} \simeq_Q \mod{\underline{\Theta} X}.$ The fact that this is a symmetric monoidal Quillen equivalence has been a vital ingredient in the construction of symmetric monoidal algebraic models, see~\cite[3.4.4]{BarnesGreenleesKedziorekShipley17} and~\cite[9.6]{PolWilliamson}.

\subsection*{Outline of the paper}
We recall the key background on model categories in Section 2, and on symmetric spectra in general model categories in Section 3. In Section 4, we recall results from Pavlov-Scholbach~\cite{PavlovScholbach18} which enable the construction of flat model structures on symmetric spectra in general model categories, and apply these results to our cases of interest. Section 5 is dedicated to the proof of Theorem~\ref{thm:main1}. In Section 6 we extend our results to show that the category of commutative $H\mathbb{Q}$-algebras is Quillen equivalent to the category of commutative rational DGAs. Finally, in Section 7 we consider the extension to modules over commutative $H\mathbb{Q}$-algebras.

\subsection*{Conventions} 
We write the left adjoint above the right adjoint in an adjoint pair displayed horizontally, and on the left in an adjoint pair displayed vertically.

\subsection*{Acknowledgements}
I am grateful to John Greenlees and Luca Pol for their comments on this paper and many helpful discussions. I would also like to thank Brooke Shipley and Sarah Whitehouse for many useful conversations and suggestions. I am also grateful to the referee for their helpful comments and suggestions on the preliminary version of this paper.

\section{Model categorical preliminaries}
In this section we recall the necessary background on model categories which we require for the paper.
\subsection{Bousfield localization}
Firstly we recall the definitions and key properties of left Bousfield localizations from~\cite{Hirschhorn03}.

\begin{defn}
	Let $\C$ be a model category and let $S$ be a collection of maps in $\C$.
	\begin{itemize}
		\item An object $W$ in $\C$ is \emph{$S$-local} if it is fibrant in $\C$ and for every $s\colon A \to B$ in $S$, the natural map $\map(B,W) \to \map(A,W)$ is a weak equivalence of homotopy function complexes.
		\item A map $f\colon X \to Y$ in $\C$ is an \emph{$S$-local equivalence} if for every $S$-local object $W$, the natural map $\map(Y,W) \to \map(X,W)$ is a weak equivalence of homotopy function complexes.
	\end{itemize}
\end{defn}

The \emph{left Bousfield localization} of $\C$ at $S$ (if it exists), denoted $L_S\C$, is the model structure on $\C$ in which the weak equivalences are the $S$-local equivalences and the cofibrations are the same as in $\C$. The fibrant objects are the $S$-local objects. We call the fibrations the $S$-local fibrations. 

The left Bousfield localization of $\C$ at $S$ exists if $S$ is a set of maps and $\C$ is left proper and cellular~\cite[4.1.1]{Hirschhorn03}, or if $S$ is a set of maps and $\C$ is left proper and combinatorial~\cite[4.7]{Barwick10}. Any weak equivalence in $\C$ is an $S$-local equivalence, so it follows that the identity functors give a Quillen adjunction $\C \rightleftarrows L_S\C.$

\begin{prop}[{\cite[3.3.16]{Hirschhorn03}}, {\cite[7.21]{JoyalTierney07}}]\label{prop:fibrationsbetweenlocals}
	Let $\C$ be a model category and $S$ a set of maps in $\C$.
	\begin{enumerate}
		\item If $f$ is an $S$-local equivalence between $S$-local objects, then $f$ is a weak equivalence in $\C$.
		\item If $f$ is a fibration between $S$-local objects, then $f$ is an $S$-local fibration.
	\end{enumerate}
\end{prop}

We now recall a result of Dugger~\cite[A.2]{Dugger01}, which when used in conjunction with Proposition~\ref{prop:fibrationsbetweenlocals} simplifies the process of proving a Quillen adjunction between left Bousfield localizations.
\begin{prop}\label{prop:Dugger}
	Let $F:\C\rightleftarrows \D:G$ be an adjunction, where $\C$ and $\D$ are model categories. Then $G$ is right Quillen if and only if $G$ preserves fibrations between fibrant objects and all acyclic fibrations.
\end{prop}

\subsection{Algebras and modules} 
We next recall the theory of (commutative) monoids, (commutative) algebras and modules in symmetric monoidal model categories due to Schwede-Shipley~\cite{SchwedeShipley00} and White~\cite{White17}.

Recall that a model category is said to be \emph{symmetric monoidal} if it has a closed symmetric monoidal structure and satisfies the following two conditions:
\begin{enumerate}
	\item \emph{pushout-product axiom}: if $f \colon A \to B$ and $g\colon X \to Y$ are cofibrations, then the pushout-product map \[f \square g \colon A \otimes Y \bigcup_{A \otimes X} B \otimes X \to B \otimes Y\] is a cofibration, which is acyclic if either $f$ or $g$ is acyclic; 
	\item \emph{unit axiom}: for $c\mathds{1} \to \mathds{1}$ a cofibrant replacement of the unit, the natural map $c\mathds{1} \otimes X \to \mathds{1} \otimes X \cong X$ is a weak equivalence for all cofibrant $X$.
\end{enumerate} 

\begin{defn}
	Suppose that $F: \C \rightleftarrows \D : U$ is a Quillen adjunction between symmetric monoidal model categories. We say that $(F,U)$ is a \emph{weak symmetric monoidal Quillen adjunction} if the right adjoint $U$ is lax symmetric monoidal (which gives the left adjoint $F$ an oplax symmetric monoidal structure) and the following conditions hold:
	\begin{enumerate}
		\item for cofibrant $A$ and $B$ in $\C$, the oplax monoidal structure map $\phi\colon F(A \otimes B) \to FA \otimes FB$ is a weak equivalence in $\D$;
		\item for a cofibrant replacement $c\mathds{1}_\C$ of the unit in $\C$, the map $F(c\mathds{1}_\C) \to \mathds{1}_\D$ is a weak equivalence in $\D$.
	\end{enumerate}
	If the oplax monoidal structure maps are isomorphisms, then we say that $(F,U)$ is a \emph{strong symmetric monoidal Quillen adjunction}. We say that $(F,U)$ is a \emph{weak (resp. strong) symmetric monoidal Quillen equivalence} if $(F,U)$ is a weak (resp. strong) symmetric monoidal Quillen adjunction which is also a Quillen equivalence. Note that if $F$ is strong monoidal and the unit of $\C$ is cofibrant, then the Quillen pair $(F,U)$ is a strong symmetric monoidal Quillen pair.
\end{defn}

In this paper, we will be particularly interested in the interaction of model structures and Quillen functors with categories of modules and (commutative) algebras. Let $(\C,\otimes,\mathds{1})$ be a symmetric monoidal model category. For a monoid $S$ in $\C$, we denote the category of (left) $S$-modules by $\mod{S}(\C)$. If the underlying category is clear, we will instead write $\mod{S}$. 

The categories of modules and algebras often inherit a model structure from the underlying category in the following way.
Let $F:\C \rightleftarrows \D: U$ be an adjunction in which $\C$ is a model category and $\D$ is a bicomplete category. Kan's lifting theorem~\cite[11.3.2]{Hirschhorn03} provides conditions under which $\D$ inherits a model structure in which a map $f$ in $\D$ is a weak equivalence (resp. fibration) if and only if $Uf$ is a weak equivalence (resp. fibration) in $\C$. We call such a model structure \emph{right lifted}.

Under mild hypotheses, the categories of modules and (commutative) algebras obtain right lifted model structures. We refer the reader to~\cite[2.4]{SchwedeShipley00} for the precise smallness condition in the following theorem, and instead note that it is satisfied if $\C$ is locally presentable. Similarly, we refer the reader to~\cite[3.3]{SchwedeShipley00} and~\cite[3.1]{White17} for the definitions of the monoid axiom and commutative monoid axiom respectively.
\begin{thm}[{\cite[4.1]{SchwedeShipley00}, \cite[3.2]{White17}}]\label{thm:model structure on algebras}
	Let $\C$ be a cofibrantly generated, symmetric monoidal model category (with some smallness condition) and let $S$ be a commutative monoid in $\C$.
	\begin{enumerate}
		\item If $\C$ satisfies the monoid axiom then the categories of $S$-modules and $S$-algebras have right lifted model structures in which a map is a weak equivalence (resp. fibration) if and only if it is a weak equivalence (resp. fibration) in $\C$.
		\item If $\C$ satisfies the commutative monoid axiom and the monoid axiom, then the category of commutative $S$-algebras has a right lifted model structure in which a map is a weak equivalence (resp. fibration) if and only if it is a weak equivalence (resp. fibration) in $\C$.
	\end{enumerate}
\end{thm}

We say that a symmetric monoidal model category $\C$ satisfies \emph{Quillen invariance of modules} if for any weak equivalence $\theta\colon S \to R$ of monoids in $\C$, the extension-restriction of scalars adjunction
\begin{center}
	\begin{tikzcd}
	\mod{S} \arrow[rr, yshift=1mm, "R \otimes_S -"] & & \mod{R} \arrow[ll, yshift=-1mm, "\res"]
	\end{tikzcd}
\end{center}
is a Quillen equivalence, see~\cite[4.3]{SchwedeShipley00}. Throughout we write $\ext = R \otimes_S -$ for the left adjoint of the restriction of scalars functor $\res$.

\subsection{Cofibrations of modules and (commutative) algebras}
In general there is not an explicit description of the cofibrations in a right lifted model structure, but in many situations they have desirable properties.
\begin{thm}[{\cite[4.1]{SchwedeShipley00}}] 
	Let $\C$ be a symmetric monoidal model category and let $S$ be a commutative monoid in $\C$. Every cofibration of $S$-algebras whose source is cofibrant as an $S$-module is also a cofibration of $S$-modules. In particular, if the unit of $\C$ is cofibrant, then every cofibrant $S$-algebra is a cofibrant $S$-module.
\end{thm}

The case of commutative algebras is more subtle. White~\cite[3.5, 3.6]{White17} has given an answer to this question in general, but it requires stronger assumptions that just the existence of the model structure on commutative algebras. We recall some relevant examples.

\begin{eg}\label{ex:convenient}
	If $S$ a commutative DGA over a field of characteristic zero and $R$ is a cofibrant commutative $S$-algebra, then $R$ is cofibrant (i.e., dg-projective) as an $S$-module, see for instance~\cite[\S 5.1]{White17}. Note that it fails in non-zero characteristic since Maschke's theorem does not apply.
\end{eg}

\begin{eg}\label{ex:notconvenient}
	In categories of spectra the situation is more complicated. It is well known by Lewis' obstruction~\cite{Lewis91} that the stable model structure on (symmetric) spectra cannot be right lifted to a model structure on commutative algebra spectra as the sphere spectrum is cofibrant. Instead, one must consider the \emph{positive} stable model structure in which the sphere spectrum is not cofibrant. This model structure can be right lifted to give a model structure on commutative algebras, however, a cofibrant commutative algebra in the positive stable model structure on spectra need not be cofibrant as a module. Nonetheless there is a model structure on spectra called the flat model structure, for which this property is true, see Corollary~\ref{cor:flatcofibrants}.  
\end{eg}

\section{Symmetric spectra in general model categories}
In this section we recall the definition of the category of symmetric spectra in general model categories and its properties and stable model structure as in~\cite{Hovey01}; see also~\cite[\S 2]{RichterShipley17}.

Let $(\C,\otimes,\mathds{1})$ be a bicomplete, closed symmetric monoidal category and $K \in \C$. Let $\Sigma$ be the category whose objects are the finite sets $\underline{n} = \{1,\ldots ,n\}$ for $n \geq 0$ where $\underline{0} = \emptyset$, and whose morphisms are the bijections of $\underline{n}$. The category of symmetric sequences in $\C$ is the functor category $\C^\Sigma$. The category $\C^\Sigma$ inherits a closed symmetric monoidal structure from $\C$ via the Day convolution, with tensor product given by \[(A \odot B)(n) = \coprod_{p+q=n}\Sigma_n \times_{\Sigma_p \times \Sigma_q} A(p) \otimes B(q).\]

The category of \emph{symmetric spectra} $\text{Sp}^\Sigma(\C,K)$ is the category of modules over $\text{Sym}(K)$ in $\C^\Sigma$, where $\text{Sym}(K) = (\mathds{1}, K, K^{\otimes 2}, \cdots)$ is the free commutative monoid on $K$. Therefore, $\symsp(\C,K)$ inherits a closed symmetric monoidal structure with tensor product defined by the coequalizer \[X \wedge Y = \mathrm{coeq}\left(X \odot \mathrm{Sym}(K) \odot Y \rightrightarrows X \odot Y\right)\] of the actions of $\mathrm{Sym}(K)$ on $X$ and $Y$. More explicitly, an object $X$ of $\text{Sp}^\Sigma(\C,K)$ is a collection of $\Sigma_n$-objects $X(n) \in \C$ with $\Sigma_n$-equivariant maps \[K \otimes X(n) \to X(n+1)\] for all $n \geq 0$, such that the composite \[K^{\otimes m} \otimes X(n) \to X(n+m)\] is $\Sigma_m \times \Sigma_n$-equivariant for all $m,n \geq 0$. Note that taking $\C = \sset_*$ and $K=S^1$ recovers the usual notion of symmetric spectra as defined and studied by Hovey-Shipley-Smith~\cite{HoveyShipleySmith00}.

We now sketch the construction of the stable model structure on $\symsp(\C,K)$ due to Hovey~\cite{Hovey01}. If $\C$ is a left proper and cellular model category, one can equip $\symsp(\C,K)$ with a level model structure in which the weak equivalences and fibrations are levelwise weak equivalences and levelwise fibrations in $\C$ respectively~\cite[8.2]{Hovey01}. One can then left Bousfield localize this level model structure to obtain the stable model structure~\cite[8.7]{Hovey01}. We call the weak equivalences in this model structure the stable equivalences and the fibrations the stable fibrations.

There is also a positive stable model structure, which allows the construction of right lifted model structures on commutative algebras, see for instance~\cite[\S 14]{MandellMaySchwedeShipley01}. However, these model structures do not have the property that cofibrant commutative algebras are cofibrant modules. In order to rectify this, we turn to the flat model structure in the next section.

\begin{notation}
	We set notation for the categories of symmetric spectra of interest.
	\begin{itemize}
		\item We write $\symsp = \symsp(\sset_*,S^1)$ for the category of symmetric spectra in simplicial sets.
		\item We write $\symsp(\sq)$ for the category  $\symsp(\sq,\widetilde{\mathbb{Q}}S^1)$ where $\sq$ is the category of simplicial $\mathbb{Q}$-modules and $\widetilde{\mathbb{Q}}\colon \sset_* \to \sq$ is the functor which takes the levelwise free $\mathbb{Q}$-module on the non-basepoint simplices.
		\item We write $\symsp(\chq)$ for the category $\symsp(\chq, \mathbb{Q}[1])$ where $\chq$ is the category of non-negatively graded chain complexes of $\mathbb{Q}$-modules and $\mathbb{Q}[1]$ is the chain complex which contains a single copy of $\mathbb{Q}$ concentrated in degree 1.
	\end{itemize}
\end{notation}

\section{Flat model structures}\label{sec:flat}
In this section we show that the categories used in Shipley's algebraicization theorem support a flat model structure. Recall from Example~\ref{ex:notconvenient} that a cofibrant commutative algebra need not be a cofibrant module in the stable model structure on spectra. To rectify this, Shipley~\cite{Shipley04} constructs a flat (and a positive flat) model structure on symmetric spectra in simplicial sets in which this property holds. Pavlov-Scholbach~\cite{PavlovScholbach18} extended these flat model structures to symmetric spectra in general model categories. The flat model structure has the same weak equivalences as the stable model structure on spectra (i.e., the stable equivalences), but has more cofibrations. In particular, the identity functor from the stable model structure to the flat model structure is a left Quillen equivalence.

\subsection{Equivariant model structures}
The stable model structure on symmetric spectra disregards the actions of the symmetric groups on each level. Instead, the flat model structure proceeds by remembering this equivariance and building it into the model structure. There are two extreme cases: the naive case is where no equivariance is recorded and the genuine case is when all equivariance is recorded.  The flat model structure on $\text{Sp}^\Sigma(\C,K)$ (when it exists) is built from the blended model structure on $G$-objects in $\C$ which is intermediate between the naive and genuine structures. Note that some authors refer to this model structure as the mixed model structure, but we do not since it is not mixed in the sense of Cole mixing~\cite{Cole06}. 

From now on, we assume that $\C$ is a pretty small model category~\cite[2.1]{PavlovScholbach18b}. We note that this condition is satisfied for simplicial sets, simplicial $\mathbb{Q}$-modules and non-negatively graded chain complexes of $\mathbb{Q}$-modules. 

We now recall the conditions needed for the genuine and blended model structures to exist, see for instance~\cite{Stephan16}. Let $G$ be a finite group. We write $G\C$ for the category of $G$-objects in $\C$; that is, the functor category $[BG,\C]$ where $BG$ is the one-object category whose morphisms are elements of $G$.
\begin{defn}
	We say that $\C$ satisfies the \emph{weak cellularity conditions for $G$} if the following are true for all subgroups $H,K \leq G$:
	\begin{enumerate}
		\item $(-)^H$ preserves directed colimits of diagrams in $G\C$ where each underlying arrow in $\C$ is a cofibration,
		\item $(-)^H$ preserves pushouts of diagrams where one leg is of the form $G/K \otimes f$ for $f$ a cofibration in $\C$,
		\item $(G/K \otimes -)^H$ takes generating cofibrations to cofibrations and generating acyclic cofibrations to acyclic cofibrations.
	\end{enumerate}
	We say that it satisfies the \emph{strong cellularity conditions for $G$} if (1) and (2) from above hold, and for any $H,K \leq G$ and any $X \in \C$, \[(G/H \otimes X)^K \cong (G/H)^K \otimes X.\]
\end{defn}

\begin{defn}
	We say that a map $f\colon X \to Y$ in $G\C$ is:
	\begin{itemize}
		\item a \emph{naive weak equivalence} if the underlying morphism is a weak equivalence in $\C$;
		\item a \emph{naive fibration} if the underlying morphism is a fibration in $\C$;
		\item a \emph{naive cofibration} if it has the left lifting property with respect to the naive acyclic fibrations;
		\item a \emph{genuine weak equivalence} if for every subgroup $H$ of $G$, the map $f^H\colon X^H \to Y^H$ is a weak equivalence in $\C$;
		\item a \emph{genuine fibration} if for every subgroup $H$ of $G$, the map $f^H\colon X^H \to Y^H$ is a fibration in $\C$;
		\item a \emph{genuine cofibration} if it has the left lifting property with respect to all genuine acyclic fibrations.
		\item a \emph{blended fibration} if it has the right lifting property with respect to maps which are both naive weak equivalences and genuine cofibrations.
	\end{itemize}
\end{defn}

The cellularity conditions control when the genuine model structure on $G\C$ exists.
\begin{prop}\label{prop:genuinemodel}
	If the weak cellularity conditions hold for $\C$ then the genuine weak equivalences, genuine cofibrations and genuine fibrations give a cofibrantly generated, model structure on $G\C$ called the genuine model structure. Furthermore, if $\C$ is proper, then so is the genuine model structure on $G\C$, and if $\C$ is a monoidal model category with cofibrant unit, then so is the genuine model structure on $G\C$.
\end{prop}
\begin{proof}
	The claim that the genuine model structure exists and is cofibrantly generated is due to Stephan~\cite[2.6]{Stephan16}. The generating cofibrations and generating acyclic cofibrations are given by $\cup_{H \leq G}\{G/H \otimes i \mid i \in I\}$ and $\cup_{H \leq G}\{G/H \otimes j \mid j \in J\}$ respectively, where $I$ and $J$ are the sets of generating cofibrations and acyclic cofibrations for $\C$ respectively.
	
	We now prove that the genuine model structure is left proper. It suffices to prove that in a diagram of pushouts of the form 
	\begin{center}
		\begin{tikzcd}
		G/H \otimes A \arrow[d, hookrightarrow] \arrow[r] \arrow[rd, phantom, very near end, "\pushout"] & C \arrow[r, "\sim"] \arrow[d] \arrow[rd, phantom, very near end, "\pushout"] & X \arrow[d] \\
		G/H \otimes B \arrow[r] & D \arrow[r] & Y 
		\end{tikzcd}
	\end{center}
	where $A \to B$ is a generating cofibration for $\C$, the map $D \to Y$ is a genuine weak equivalence. This is because as $\C$ is pretty small, weak equivalences are closed under transfinite composition~\cite[2.2]{PavlovScholbach18b}, and therefore the class of maps for which pushing out along them preserves weak equivalences is closed under retracts, pushouts and transfinite compositions. By the second cellularity condition, after taking $K$ fixed points, the left hand square and the outer rectangle are still pushouts. It follows that the right hand square is also still a pushout. 
	
	By the third cellularity condition, the left most vertical map is still a cofibration after taking $K$ fixed points. 
	Since cofibrations are stable under pushout, the map $C^K \to D^K$ is a cofibration, and since $\C$ is left proper, we have that $D^K \to Y^K$ is a weak equivalence for all $K$. Hence the genuine model structure is left proper.
	The fact that the model structure is right proper follows immediately from the fact that fixed points determine fibrations and weak equivalences.
	
	We now prove that the genuine model structure is monoidal. Firstly we must show that the pushout-product of two genuine cofibrations is a genuine cofibration. We use the description of the generating cofibrations $\cup_{H \leq G}\{G/H \otimes i \mid i \in I\}$ for the genuine model structure where $I$ is the set of generating cofibrations for $\C$. Take generating cofibrations $$G/H \otimes i\colon G/H \otimes A \to G/H \otimes B \quad \text{and} \quad G/K \otimes i'\colon G/K \otimes X \to G/K \otimes Y$$ for the genuine model structure. Since $G/H \otimes -$ and $G/K \otimes -$ are left adjoints, the pushout product map $(G/H \otimes i) \square (G/K \otimes i')$ can be identified with the map $(G/H \otimes G/K) \otimes (i \square i')$, which in turn can be identified with $$\coprod_{x \in [H\backslash G/K]}G/(H \cap xKx^{-1}) \otimes (i \square i')$$ by the double coset formula. Since $\C$ is monoidal, $i \square i'$ is a cofibration in $\C$ and hence the pushout product map $(G/H \otimes i) \square (G/K \otimes i')$ is a genuine cofibration as required. It follows by a similar argument that the pushout product of a genuine cofibration with a genuine acyclic cofibration is a genuine acyclic cofibration.  
	
	For the unit axiom, note that the monoidal unit in $G\C$ is the unit of $\C$ equipped with the trivial $G$-action. The functor which equips an object with the trivial $G$-action is left adjoint to the $G$-fixed points functor, and hence is left Quillen. It then follows that since the unit of $\C$ is cofibrant, the unit in $G\C$ is genuine cofibrant.
\end{proof}

We can then localize the genuine model structure to give the blended model structure. 
\begin{thm}\label{thm:blended}
	Let $\C$ be a simplicial, proper model category which satisfies the weak cellularity conditions. Then the naive weak equivalences, genuine cofibrations and blended fibrations give a proper, cofibrantly generated model structure on $G\C$ which we call the blended model structure.
\end{thm}
\begin{proof}
	We apply Bousfield-Friedlander localization~\cite[9.3]{Bousfield01} to the genuine model structure on $G\C$, with $QX = \map (EG, \widehat{f}X)$ where $\widehat{f}$ is a genuine fibrant replacement functor and $\map$ denotes the simplicial cotensor. We must verify that the conditions (A1), (A2) and (A3) from~\cite[9.3]{Bousfield01} are satisfied. Note that a map $f\colon X \to Y$ in $G\C$ is a naive weak equivalence if and only if $Qf\colon QX \to QY$ is a genuine weak equivalence. To see this, if $f\colon X \to Y$ is a naive weak equivalence, then $\map(G,f)$ is a genuine weak equivalence. Therefore, $\map(Z,f)$ is a genuine weak equivalence if $Z$ is built from free cells. Conversely, since $EG \to *$ is a naive weak equivalence, if $\map(EG,f)$ is a genuine weak equivalence then $f$ is a naive weak equivalence. The conditions (A1) and (A2) follow from this observation. Since $Q$ preserves fibrations and pullbacks, condition (A3) follows from the right properness of the genuine model structure on $G\C$.
\end{proof}

Note that~\cite[9.3]{Bousfield01} also gives an explicit description of the blended fibrations as those maps $X \to Y$ which are genuine fibrations and have the property that
\begin{equation}\label{homotopy pullback}
\begin{tikzcd}
X \arrow[r] \arrow[d] & \map(EG,\widehat{f}X) \arrow[d] \\
Y \arrow[r] & \map(EG, \widehat{f}Y) 
\end{tikzcd}\tag{$\star$}
\end{equation}
is a homotopy pullback square. The genuine fibrant replacement ensures that this is equivalent to the square being a homotopy pullback after taking $H$-fixed points for all $H \leq G$. 

\begin{prop}
	The blended model structure exists on $G\C$ for $\C =\sset_*$, $\sq$ and $\chq$. 
\end{prop}
\begin{proof}
	The categories of based simplicial sets and simplicial $\mathbb{Q}$-modules satisfy the strong cellularity conditions by~\cite[2.14]{Stephan16}. The category of non-negatively graded rational chain complexes satisfies the weak cellularity conditions by~\cite[2.19]{Stephan16}. Therefore the result follows from Theorem~\ref{thm:blended}.
\end{proof}

Finally we note that in these cases, the blended model structure can be identified with the injective model structure in which the weak equivalences and cofibrations are both underlying.
\begin{prop}\label{prop:injective}\leavevmode
\begin{itemize}
	\item[(i)] A map $f$ in $G\text{-}\sset_*$ is an underlying cofibration if and only if it is a genuine cofibration.
	\item[(ii)] For $\C = \sq$ and $\chq$, a map $f$ in $G\C$ is an underlying cofibration if and only if it is a naive cofibration if and only if it is a genuine cofibration.
\end{itemize}
\end{prop}
\begin{proof}
	Part (i) is well known; see for example~\cite[1.2]{Shipley04} or~\cite[2.16]{Stephan16}.
	
	For part (ii), let $\C = \chq$ or $\sq$ and note that we can give the same style of proof since they are both $\mathbb{Q}$-additive. From the description of the generating cofibrations of the genuine model structure given in the proof of Proposition~\ref{prop:genuinemodel}, it is clear that any genuine cofibration is an underlying cofibration. Since any naive cofibration is also a genuine cofibration it follows that any naive cofibration is an underlying cofibration. 
	
	We now turn to proving the forward implication. Since any naive cofibration is a genuine cofibration, it suffices to show that if $f$ is an underlying cofibration then it is a naive cofibration.
Let $f\colon X \to Y$ be an underlying cofibration in $G\C$. In order to prove that $f$ is a naive cofibration we must show that it has the left lifting property with respect to the naive acyclic fibrations. Consider a commutative square
	\[\begin{tikzcd}
	X \arrow[r, "\alpha"] \arrow[d, "f"'] & A \arrow[d, "h"] \\
	Y \arrow[r, "\beta"'] & B 
	\end{tikzcd} \]
in $G\C$, in which $h$ is a naive acyclic fibration. Since $f$ is an underlying cofibration, there is a lift $\theta\colon Y \to A$ making the diagram commute, but this need not be an equivariant map. Define $\phi\colon Y \to A$ by $$\phi(y) = \frac{1}{|G|}\sum_{g \in G}g\theta(g^{-1}y).$$ This is an equivariant map, so it remains to check that it is indeed a lift. 

Since $f$ and $\alpha$ are equivariant maps, $$\phi(f(x)) = \frac{1}{|G|}\sum_{g \in G}g\theta(f(g^{-1}x)) = \frac{1}{|G|}\sum_{g \in G}g\alpha (g^{-1}x) = \alpha(x).$$ In a similar way, one can show that $h\phi = \beta$. Therefore $\phi$ is a lift, and the map $f$ is a naive cofibration and hence also a genuine cofibration.
\end{proof}

\begin{cor}\label{cor:naive and blended}
For $\C = \sq$ and $\chq$, the blended model structure, injective model structure and the naive model structure on $G\C$ are the same.
\end{cor}
\begin{proof}
The weak equivalences in all three model structures are the naive weak equivalences. The cofibrations in each coincide by Proposition~\ref{prop:injective}.
\end{proof}

We emphasize that in the case of $\C = \sset_*$, the blended model structure is the same as the injective model structure on $G\C$, but is \emph{not} the same as the naive model structure.

\begin{cor}
For $\C = \sset_*$, $\sq$ and $\chq$, the blended model structure on $G\C$ is monoidal.
\end{cor}
\begin{proof}
Note that in each case, $\C$ is monoidal and has cofibrant unit. Since the blended model structure is the same as the injective model structure by Proposition~\ref{prop:injective}, it is immediate that the pushout-product axiom holds. The unit axiom holds by the same argument as in Proposition~\ref{prop:genuinemodel}. 
\end{proof}

\subsection{The flat model structure}
We can equip $\symsp$, $\symsp(\sq)$ and $\symsp(\chq)$ with the \emph{level flat model structure}, in which the weak equivalences (resp. fibrations) are the levelwise naive weak equivalences (resp. levelwise blended fibrations)~\cite[3.1.3]{PavlovScholbach18}. The cofibrations in the level flat model structure are the \emph{flat cofibrations}; that is, the maps which have the left lifting property with respect to maps which are both levelwise naive weak equivalences and levelwise blended fibrations. In a similar manner, $\symsp$, $\symsp(\sq)$ and $\symsp(\chq)$ can be given the \emph{positive level flat model structure} in which the weak equivalences (resp. fibrations) are the maps which are naive weak equivalences (resp. blended fibrations) for each level $n > 0$.

A left Bousfield localization of the level flat model structure yields the \emph{flat model structure}. The weak equivalences in the flat model structure are the stable equivalences, and the cofibrations are the flat cofibrations. We call the fibrations in the flat model structure the \emph{flat fibrations}. Similarly, a left Bousfield localization of the positive level flat model structure gives the \emph{positive flat model structure} in which the weak equivalences are also the stable equivalences.
\begin{thm}
	The flat and positive flat model structures on $\symsp$, $\symsp(\sq)$ and $\symsp(\chq)$ (and on modules over monoids in these categories) exist. Furthermore, they satisfy Quillen invariance of modules, and are stable, left proper, symmetric monoidal and combinatorial model structures. 
\end{thm}
\begin{proof}
	Since the genuine cofibrations are the same as the underlying cofibrations by Proposition~\ref{prop:injective}, the blended model structure coincides with the injective model structure. The injective model structure is strongly admissible by~\cite[2.3.7]{PavlovScholbach18} and therefore the flat model structure exists by~\cite[3.2.1]{PavlovScholbach18}. Quillen invariance holds by~\cite[3.3.9]{PavlovScholbach18}, monoidality follows as it is defined to be a monoidal left Bousfield localization, stability by~\cite[3.4.1]{PavlovScholbach18} and left properness and combinatoriality follows from~\cite[3.4.2]{PavlovScholbach18}.
\end{proof}

We now record some key properties of the flat model structure which we will use throughout this paper.
\begin{prop}\label{prop:omnibus}\leavevmode
\begin{itemize}
\item[(i)] A map is an acyclic flat fibration if and only if it is a levelwise acyclic flat fibration.
\item[(ii)] A map between flat fibrant objects is a flat fibration if and only if it is a levelwise flat fibration.
\item[(iii)] The identity functor is a left Quillen equivalence from the stable model structure to the flat model structure.
\end{itemize}
\end{prop}
\begin{proof}
Part (i) follows from the fact that left Bousfield localization does not change the acyclic fibrations and part (ii) follows from Proposition~\ref{prop:fibrationsbetweenlocals}. For part (iii), since the stable model structure and the flat model structure have the same weak equivalences, it suffices to show that any stable cofibration is a flat cofibration. A map is a stable cofibration if and only if it has the left lifting property with respect to maps which are levelwise naive acyclic fibrations, and a map is a flat cofibration if and only if it has the left lifting property with respect to maps which are both levelwise naive weak equivalences and blended fibrations. Any blended fibration is a naive fibration, and therefore a stable cofibration is also a flat cofibration.
\end{proof}

\begin{cor}\label{cor:flatcofibrants}
	Let $S$ be a commutative monoid in $\symsp$, $\symsp(\sq)$ or $\symsp(\chq)$.
	The positive flat model structure can be right lifted to give a model structure on commutative $S$-algebras. Moreover, a positively flat cofibrant commutative $S$-algebra is also flat cofibrant as an $S$-module.
\end{cor}
\begin{proof}
	Since the blended model structure coincides with the injective model structure in these cases by Proposition~\ref{prop:injective}, and the injective model structure is strongly admissible~\cite[2.3.7]{PavlovScholbach18}, this is a consequence of~\cite[4.1, 4.4]{PavlovScholbach18}.
\end{proof}

The flat model structure is a left Bousfield localization of the level flat model structure where weak equivalences and fibrations are determined levelwise in the blended model structure. We can give a characterization of the fibrant objects in the flat model structure. 
\begin{prop}[{\cite[3.2.1]{PavlovScholbach18}}]\label{prop:fibrant}
	An object $X$ of $\symsp(\C,K)$ is flat fibrant if and only if $X$ is level flat fibrant and $X_n \to \underline{\mathrm{Hom}}(K,X_{n+1})$ is a naive weak equivalence where $\underline{\mathrm{Hom}}(K,-)$ is the right adjoint to $K \otimes -$.
\end{prop}

The following corollary shows that stable model structures on $\symsp(\sq)$ and $\symsp(\chq)$ satisfy extra compatibility between commutative algebras and modules, unlike the stable model structure on $\symsp$.
\begin{cor}\label{cor:stable and flat}
The flat (resp. positive flat) model structure on $\symsp(\sq)$ and $\symsp(\chq)$ is the same as the stable (resp. positive stable) model structure.
\end{cor}
\begin{proof}
The weak equivalences in both the flat and stable model structure are the stable equivalences. Therefore it suffices to show that they have the same acyclic fibrations. A map is an acyclic fibration in the stable model structure if and only if it is a levelwise naive acyclic fibration. By Corollary~\ref{cor:naive and blended}, this is the case if and only if it is a levelwise acyclic fibration in the blended model structure, i.e., an acyclic flat fibration.
\end{proof}
In light of the previous corollary, we could call the model structure which we use on $\symsp(\sq)$ and $\symsp(\chq)$ either flat or stable. However, we will often refer to it as the flat model structure to remind the reader of the extra compatibility between commutative algebras and modules given by Corollary~\ref{cor:flatcofibrants}, which we will use throughout the paper.

\section{Shipley's algebraicization theorem in the flat setting}\label{sec:shipley1}
In this section we show that the chain of Quillen equivalences given by Shipley~\cite{Shipley07} for the stable model structure are still Quillen equivalences in the flat model structure, for the rational case. The identity functor from the stable model structure to the flat model structure is a left Quillen equivalence by Proposition~\ref{prop:omnibus}. Therefore by the 2-out-of-3 property of Quillen equivalences, it is sufficient to check that we get Quillen adjunctions in the flat model structure. In fact, since the stable and flat model structure are the same on $\symsp(\sq)$ and $\symsp(\chq)$ by Corollary~\ref{cor:stable and flat}, this reduces to just checking that the first adjunction is a Quillen adjunction. The following diagram summarises all of the adjunctions between $H\mathbb{Q}$-modules and chain complexes of $\mathbb{Q}$-modules. \begin{equation}\label{diagram}
\begin{tikzcd}
\mod{H\mathbb{Q}}^\mathrm{flat} \arrow[r, yshift=2mm, "Z"] \arrow[dd, xshift=4mm, "1"] \arrow[dd, phantom, "\simeq_Q"] & \symsp(\sq)_{\mathrm{flat}} \arrow[dd, xshift=4mm, "1"] \arrow[r, yshift=-2mm, "\phi^* N"'] \arrow[l,yshift=-2mm, "U"] \arrow[dd, phantom, "{=}"] & \symsp(\chq)_{\mathrm{flat}}  \arrow[dd, xshift=4mm, "1"] \arrow[dd, phantom, "{=}"] \arrow[l, yshift=2mm, "L"'] \arrow[r, yshift=2mm, "D"] &  \arrow[dd, xshift=4mm, "1"] \arrow[dd, phantom, "{=}"] \mathrm{Ch}_\mathbb{Q} \arrow[l, yshift=-2mm, "R"] \\ & & & \\ \mod{H\mathbb{Q}}^\mathrm{stable} \arrow[r, yshift=3mm, "Z"] \arrow[uu, xshift=-4mm, "1"]  \arrow[r, phantom, "\simeq_Q"] &  \symsp(\sq)_{\mathrm{stable}} \arrow[l,yshift=-3mm, "U"] \arrow[uu, xshift=-4mm, "1"] \arrow[r, yshift=-3mm, "\phi^* N"'] \arrow[r, phantom, "\simeq_Q"] & \symsp(\chq)_{\mathrm{stable}} \arrow[uu, xshift=-4mm, "1"] \arrow[l, yshift=3mm, "L"'] \arrow[r, yshift=3mm, "D"] \arrow[r, phantom, "\simeq_Q"] & \mathrm{Ch}_\mathbb{Q} \arrow[uu, xshift=-4mm, "1"] \arrow[l, yshift=-3mm, "R"] 
\end{tikzcd}\tag{$\dagger$}
\end{equation}
The functors will be defined throughout the rest of the section.

The model structure on simplicial $\mathbb{Q}$-modules is right lifted from simplicial sets along the forgetful functor $\sq \to \sset_*$. Applying this functor levelwise gives a forgetful functor $\widetilde{U}\colon\symsp(\sq) \to \symsp$.
Note that $\widetilde{U}\text{Sym}(\widetilde{\mathbb{Q}}S^1) = (\mathbb{Q}, \widetilde{\mathbb{Q}}S^1, \widetilde{\mathbb{Q}}S^2,...)$ which is $H\mathbb{Q}$~\cite[1.2.5]{HoveyShipleySmith00}.  Therefore the forgetful functor $\widetilde{U}$ can be viewed as a functor $U\colon\symsp(\sq) \to \mod{H\mathbb{Q}}$. 

Firstly, we show that the forgetful functor $\widetilde{U}$ is right Quillen when $\Sp^\Sigma$ is equipped with the flat model structure. Even though the flat model structure and the stable model structure on $\symsp(\sq)$ are the same by Corollary~\ref{cor:stable and flat}, in order to prove the following it is actually convenient to work with the description of the acyclic fibrations in the flat model structure.
\begin{lem}\label{lem:Udetects}
	The forgetful functor \[\widetilde{U}\colon\symsp(\sq)_\mathrm{flat} \to \symsp_\mathrm{flat}\] preserves fibrations, and preserves and detects weak equivalences.
\end{lem}
\begin{proof}
	The forgetful functor preserves and detects weak equivalences by~\cite[Proof of 4.1]{Shipley07}. We now show that it preserves the fibrations. By Proposition~\ref{prop:Dugger} it is sufficient to show that $\widetilde{U}$ preserves the acyclic flat fibrations and the flat fibrations between flat fibrant objects.
	
	 A map is an acyclic flat fibration if and only if it is a levelwise acyclic flat fibration, so it suffices to show that the forgetful functor $\sq \to \sset_*$ preserves naive weak equivalences and blended fibrations. Since the model structure on $\sq$ is right lifted from $\sset_*$, the forgetful functor preserves naive weak equivalences and genuine fibrations. It remains to check the homotopy pullback condition~(\ref{homotopy pullback}), which is an immediate consequence of the fact that the forgetful functor preserves homotopy pullbacks.
	 
	 A flat fibration between flat fibrant objects is a levelwise flat fibration and hence $\widetilde{U}$ sends it to a levelwise flat fibration by the previous paragraph. Therefore, it remains to show that $\widetilde{U}$ preserves flat fibrant objects. Let $X$ be a flat fibrant object in $\symsp(\sq)_\mathrm{flat}$. By Proposition~\ref{prop:fibrant}, $X$ is level flat fibrant and $X_n \to \underline{\mathrm{Hom}}(\widetilde{\mathbb{Q}}S^1, X_{n+1})$ is a naive weak equivalence. It follows that $\widetilde{U}X$ is level flat fibrant, and since $\widetilde{U}$ preserves naive weak equivalences, we also have that $\widetilde{U}X_n \to \widetilde{U}\underline{\mathrm{Hom}}(\widetilde{\mathbb{Q}}S^1, X_{n+1})$ is a naive weak equivalence. By the $\widetilde{\mathbb{Q}} \dashv U$ adjunction, it follows that $$\widetilde{U}X_n \to \underline{\mathrm{Hom}}(S^1, \widetilde{U}X_{n+1})$$ is a naive weak equivalence, and hence by Proposition~\ref{prop:fibrant}, $\widetilde{U}X$ is flat fibrant. Therefore, $\widetilde{U}$ preserves flat fibrant objects.
\end{proof}

\begin{cor}\label{cor:Udetects}
	The forgetful functor \[U\colon \symsp(\sq)_\mathrm{flat} \to \mod{H\mathbb{Q}}^\mathrm{flat}\] preserves fibrations, and preserves and detects weak equivalences.
\end{cor}

Recall from~\cite[4.3]{Shipley07} that the forgetful functor $U\colon\symsp(\sq) \to \mod{H\mathbb{Q}}$ has a left adjoint $Z$ defined by \[Z(X) = H\mathbb{Q} \otimes_{\widetilde{\mathbb{Q}}H\mathbb{Q}} \widetilde{\mathbb{Q}}X\] where $H\mathbb{Q}$ is viewed as a $\widetilde{\mathbb{Q}}H\mathbb{Q}$-module via the ring map $\beta\colon \widetilde{\mathbb{Q}}H\mathbb{Q} \to H\mathbb{Q}$ given by the monad structure on $\widetilde{\mathbb{Q}}$.

\begin{prop}\label{prop:step1}
	The adjunction 
	\begin{center}
		\begin{tikzcd}
		\mod{H\mathbb{Q}}^\mathrm{flat} \arrow[r, yshift=1mm, "Z"] &   \symsp(\sq)_{\mathrm{flat}} \arrow[l,yshift=-1mm, "U"]
		\end{tikzcd}
	\end{center}
	is a strong symmetric monoidal Quillen equivalence with the respect to the flat model structures.
\end{prop}
\begin{proof}
	The forgetful functor $U$ preserves weak equivalences and fibrations in the flat model structure by Corollary~\ref{cor:Udetects}. Therefore, $Z \dashv U$ is a Quillen adjunction and hence by the 2-out-of-3 property of Quillen equivalences, is a Quillen equivalence; see Diagram~(\ref{diagram}). It is a strong symmetric monoidal Quillen equivalence as $Z$ is strong symmetric monoidal and the unit $H\mathbb{Q}$ is a cofibrant $H\mathbb{Q}$-module.
\end{proof}

Applying the normalization functor $N\colon \sq \to \chq$ levelwise yields a functor $$N\colon \symsp(\sq) \to \mod{\mc{N}}\left((\chq)^\Sigma\right)$$ where $\mc{N} = N(\text{Sym}(\widetilde{\mathbb{Q}}S^1))$. There is a ring map $\phi\colon \text{Sym}(\mathbb{Q}[1]) \to \mc{N}$ induced levelwise by the lax symmetric monoidal structure on $N$, and therefore composing $N$ and $\res[\phi]$ gives a functor $\res[\phi]N\colon \symsp(\sq) \to \symsp(\chq)$. This functor has a left adjoint denoted $L$ by~\cite[\S 3.3]{SchwedeShipley03}. It is important to note that the left adjoint is not just the composite of the left adjoints of $N$ and $\res[\phi]$. Shipley~\cite[4.4]{Shipley07} shows that \begin{center}
		\begin{tikzcd}
		\symsp(\sq) \arrow[r, yshift=-1mm, "\phi^* N"'] & \symsp(\chq) \arrow[l, yshift=1mm, "L"']
		\end{tikzcd}
	\end{center}
is a weak symmetric monoidal Quillen equivalence.

The final step is the passage from symmetric spectra in non-negatively graded chain complexes to unbounded chain complexes. The inclusion $\chq \to \mathrm{Ch}_\mathbb{Q}$ of non-negatively graded chain complexes into unbounded complexes has a right adjoint $C_0$ called the connective cover. This is defined by $(C_0X)_n = X_n$ for $n \geq 1$ and $(C_0X)_0 = \mathrm{cycles}(X_0)$. 
Using the connective cover, one defines a functor $R\colon \mathrm{Ch}_\mathbb{Q} \to \symsp(\chq)$ by $(RY)_n = C_0(Y \otimes\mathbb{Q}[n])$. Recall from~\cite{Shipley07} that this functor has a left adjoint $D$. Moreover, $D$ is strong symmetric monoidal as proved by Strickland~\cite{Strickland}. Note that this fact has been subject to some confusion, see~\cite{Shipleycorrection}. Shipley~\cite[4.7]{Shipley07} shows that \begin{center}
		\begin{tikzcd}
		\symsp(\chq) \arrow[r, yshift=1mm, "D"] & \mathrm{Ch}_\mathbb{Q} \arrow[l, yshift=-1mm, "R"]
		\end{tikzcd}
	\end{center}
is a strong symmetric monoidal Quillen equivalence where $\mathrm{Ch}_\mathbb{Q}$ is equipped with the projective model structure.

Combining the results of this section gives a proof of Theorem~\ref{thm:main1}.

\section{Extension to commutative algebras}
Let $F:\C \rightleftarrows \D:G$ be a weak symmetric monoidal Quillen pair. As $G$ is lax symmetric monoidal, it preserves commutative monoids and therefore gives rise to a functor $G\colon \mathrm{CMon}(\D) \to \mathrm{CMon}(\C)$. If the Quillen pair is a \emph{strong} symmetric monoidal Quillen pair, then $F$ also lifts to a functor on commutative monoids. However, when $F$ is only oplax symmetric monoidal, it will not necessarily preserve commutative monoids. 

We always equip the category of commutative monoids with the model structure right lifted along the forgetful functor, see Theorem~\ref{thm:model structure on algebras}. The forgetful functor $U\colon \mathrm{CMon}(\C) \to \C$ has a left adjoint given by \[\mathbb{P}_\C (X) = \bigvee_{n \geq 0} X^{\wedge n}/\Sigma_n.\] The adjoint lifting theorem~\cite[4.5.6]{Borceux94} implies that the lift of $G$ to the categories of commutative monoids has a left adjoint $\widetilde{F}$ defined by the coequalizer diagram
\begin{center}
	\begin{tikzcd}
	\mathbb{P}_\D F\mathbb{P}_\C X \arrow[r, yshift=1mm] \arrow[r, yshift=-1mm] & \mathbb{P}_\D FX \arrow[r] & \widetilde{F}X.
	\end{tikzcd}
\end{center}
One of the maps is obtained from the counit of the $\mathbb{P}_\C \dashv U$ adjunction, and the other map is adjunct to the natural map \[F\mathbb{P}_\C X \cong \bigvee_{n \geq 0}F(X^{\wedge n})/\Sigma_n \to \bigvee_{n \geq 0}(FX)^{\wedge n}/\Sigma_n \cong \mathbb{P}_\D FX\] obtained from the oplax structure on $F$. Since $G$ preserves commutative monoids, there is a natural isomorphism $UG \cong GU$ and by adjunction there is a natural isomorphism \[\mathbb{P}_\D F \cong \widetilde{F}\mathbb{P}_\C.\]

Before we can state a theorem about lifting weak symmetric monoidal Quillen equivalences to Quillen equivalences on commutative monoids, we need to impose a hypothesis.
\begin{hyp}\label{hyp:orbits}
Let $F:\C \rightleftarrows \D : G$ be a weak symmetric monoidal Quillen equivalence. For any cofibrant object $X$ of $\C$, the natural map \[F(X^{\wedge n})/\Sigma_n \to (FX)^{\wedge n}/\Sigma_n\] is a weak equivalence in $\D$.
\end{hyp}

\begin{lem}\label{lem:orbits}
Let $F:\C \rightleftarrows \D : G$ be a weak symmetric monoidal Quillen equivalence. This satisfies Hypothesis~\ref{hyp:orbits} if either of the following conditions hold:
\begin{itemize}
\item[(i)] $F:\C \rightleftarrows \D : G$ is a strong symmetric monoidal Quillen equivalence;
\item[(ii)] underlying cofibrant objects in $\Sigma_n\D$ are naive cofibrant.
\end{itemize}
\end{lem}
\begin{proof}
The first part follows immediately from the definition. For the second part, let $X$ be cofibrant in $\C$. By definition of a weak symmetric monoidal Quillen pair, the natural map $F(X^{\wedge n}) \to (FX)^{\wedge n}$ is a weak equivalence between cofibrant objects in the injective model structure on $\Sigma_n\D$. By hypothesis, this is moreover a naive weak equivalence between naive cofibrant objects. From the description of the generating (acyclic) cofibrations given in Proposition~\ref{prop:genuinemodel} one can see that the orbits functor $(-)/\Sigma_n\colon \Sigma_n\D \to \D$ is left Quillen when $\Sigma_n\D$ is equipped with the genuine model structure. Since the identity is a left Quillen functor from the naive model structure on $\Sigma_n\D$ to the genuine model structure, it follows that $(-)/\Sigma_n\colon \Sigma_n\D \to \D$ is left Quillen when $\Sigma_n\D$ is equipped with the naive model structure. By Ken Brown's lemma, it then follows that $F(X^{\wedge n})/\Sigma_n \to (FX)^{\wedge n}/\Sigma_n$ is a weak equivalence in $\D$.
\end{proof}

We now state when weak symmetric monoidal Quillen equivalences lift to Quillen equivalences between the categories of commutative monoids. This result is closely related to work of Schwede-Shipley, White and White-Yau. Schwede-Shipley~\cite[3.12(3)]{SchwedeShipley03} consider the related question on associative monoids without the commutativity assumption, White~\cite[4.19]{White17} provides hypotheses under which \emph{strong} monoidal Quillen equivalences lift to the categories of commutative monoids and White-Yau~\cite[5.8]{WhiteYau19} provide hypotheses under which \emph{weak} monoidal Quillen equivalences lift. The most general of the statements is that of White-Yau where the result follows from a more general result about lifting Quillen equivalences to categories of coloured operads. 

For orientation in the following statement and proof, the reader might like to consider $\C$ being the positive flat model structure on spectra and $\widetilde{\C}$ being the flat model structure on spectra. The hypotheses are designed in such a way that this example fits into the framework. We note that we write left adjoint functors on the left in an adjoint pair displayed vertically.
\begin{thm}\label{thm:lift to comm mon}
	Let $F:\C \rightleftarrows \D : G$ be a weak symmetric monoidal Quillen equivalence between cofibrantly generated model categories which satisfy the commutative monoid axiom and the monoid axiom. Suppose that the underlying categories of $\C$ and $\D$ support other model structures denoted $\widetilde{\C}$ and $\widetilde{\D}$ respectively, with the same weak equivalences, such that 
\[ 
\begin{tikzcd}
\C \arrow[d, xshift=-1mm, "1"'] \arrow[r, "F", yshift=1mm] & \D \arrow[d, xshift=-1mm, "1"']  \arrow[l, "G", yshift=-1mm] \\
\widetilde{\C} \arrow[r, "F", yshift=1mm] \arrow[u, xshift=1mm, "1"'] & \widetilde{\D} \arrow[u, xshift=1mm, "1"'] \arrow[l, "G", yshift=-1mm]
\end{tikzcd}
\]	
are all Quillen adjunctions. Suppose that cofibrant commutative monoids in $\C$ (resp. $\D$) are cofibrant in $\widetilde{\C}$ (resp. $\widetilde{\D}$), the generating cofibrations $I$ of $\C$ have cofibrant source (and hence target), the monoidal unit $\mathds{1}_\C$ of $\C$ is cofibrant and that Hypothesis~\ref{hyp:orbits} is satisfied.
Then there is a Quillen equivalence \[\widetilde{F}: \mathrm{CMon}(\C) \rightleftarrows \mathrm{CMon}(\D): G.\]
\end{thm}
\begin{proof}
	Since the model structures are right lifted, $G$ preserves fibrations and acyclic fibrations and therefore is right Quillen as a functor $\mathrm{CMon}(\D) \to \mathrm{CMon}(\C)$. Let $A$ be a cofibrant commutative monoid in $\C$ and $B$ be a fibrant commutative monoid in $\D$. We must show that the map $A \to GB$ is a weak equivalence in $\C$ if and only if $\widetilde{F}A \to B$ is a weak equivalence in $\D$. 
	
	The adjunction unit of $\widetilde{F} \dashv G$ gives rise to a map $UA \to UG\widetilde{F}A \cong GU\widetilde{F}A$ and hence by adjunction there is a natural map $FA \to \widetilde{F}A$ where we neglect to write the forgetful functors. The composite $FA \to \widetilde{F}A \to B$ is adjunct to the map $A \to GB$ in $\C$. 
	
	Let $X \in \D$. Write $fX$ for a fibrant replacement of $X$ in $\D$ and $\widetilde{f}X$ for a fibrant replacement of $X$ in $\widetilde{\D}$. Consider the square
	\[\begin{tikzcd}
	X \arrow[d] \arrow[r] & \widetilde{f}X \arrow[d] \\
	fX \arrow[r] & \ast
	\end{tikzcd} \]
in which the left vertical arrow is an acyclic cofibration in $\D$, and the right vertical is a fibration in $\widetilde{\D}$ and hence in $\D$. By lifting properties, we obtain a map $fX \to \widetilde{f}X$ which is a weak equivalence. 
	
We must show that the map $A \to GB$ is a weak equivalence in $\C$ if and only if $\widetilde{F}A \to B$ is a weak equivalence in $\D$, where $A$ is cofibrant in $\mathrm{CMon}(\C)$ and $B$ is fibrant in $\mathrm{CMon}(\D)$. By the previous paragraph, we have a weak equivalence $B \to \widetilde{f}B$ where $\widetilde{f}B$ is fibrant in $\widetilde{\D}$ and hence in $\D$. By Ken Brown's lemma, $GB \to G\widetilde{f}B$ is a weak equivalence, and therefore $A \to GB$ is a weak equivalence if and only if $A \to G\widetilde{f}B$ is a weak equivalence. 

Note that since $\C$ and $\widetilde{\C}$ have the same weak equivalences, the identity functor $\C \to \widetilde{\C}$ is a left Quillen equivalence, and similarly for $\D$. Therefore, by the 2-out-of-3 property of Quillen equivalences, $F: \widetilde{\C} \rightleftarrows \widetilde{\D}: G$ is a Quillen equivalence. Since $A$ is cofibrant in $\mathrm{CMon}(\C)$ and hence in $\widetilde{\C}$, and $\widetilde{f}B$ is fibrant in $\widetilde{\D}$, $A \to G\widetilde{f}B$ is a weak equivalence if and only if $FA \to \widetilde{f}B$ is a weak equivalence. Since $B \to \widetilde{f}B$ is a weak equivalence, $FA \to \widetilde{f}B$ is a weak equivalence if and only if $FA \to B$ is a weak equivalence. Since the composite $FA \to \widetilde{F}A \to B$ is adjunct to the map $A \to GB$ in $\C$, it follows that it is enough to show that $\lambda_A\colon FA \to \widetilde{F}A$ is a weak equivalence. 
	
	As $\C$ is cofibrantly generated, $A$ is a retract of a $\mathbb{P}_\C(I)$-cell complex where $I$ is the set of generating cofibrations for $\C$, i.e., $\emptyset \to A$ is a retract of a transfinite composition of pushouts of maps in $\mathbb{P}_\C(I)$. We proceed by transfinite induction on the transfinite composition which defines a cofibrant object. The base case is the claim that $F(\mathds{1}_\C) \to \widetilde{F}(\mathds{1}_\C)$ is a weak equivalence. The left adjoint $\widetilde{F}$ takes the initial object $\mathds{1}_\C$ of $\mathrm{CMon}(\C)$ to the initial object $\mathds{1}_\D$ of $\mathrm{CMon}(\D)$. Since $\mathds{1}_\C$ is cofibrant, $F(\mathds{1}_\C) \to \mathds{1}_\D$ is a weak equivalence by the unit axiom of the weak monoidal Quillen adjunction $F \dashv G$. Therefore the base case holds.
	
	Write $\mathbb{P}^nX = X^{\wedge n}/\Sigma_n$, so that $\mathbb{P}X = \vee_{n \geq 0}\mathbb{P}^nX.$ By Hypothesis~\ref{hyp:orbits}, if $X$ is a cofibrant object of $\C$, $F(\mathbb{P}^n_\C X) = F(X^{\wedge n})/\Sigma_n \to (FX)^{\wedge n}/\Sigma_n = \mathbb{P}^n_\D(FX)$ is a weak equivalence. Since $X$ is cofibrant in $\C$, $\mathbb{P}_{\C}^nX$ is cofibrant in $\widetilde{\C}$ and therefore $F(\mathbb{P}_{\C}^nX)$ is cofibrant in $\widetilde{\D}$. In a similar way, one sees that $\mathbb{P}_{\D}^n(FX)$ is cofibrant in $\widetilde{\D}$. Therefore $F(X^{\wedge n})/\Sigma_n \to (FX)^{\wedge n}/\Sigma_n$ is a weak equivalence between cofibrant objects in $\widetilde{\D}$ and Ken Brown's lemma shows that taking coproducts preserves this weak equivalence. Therefore, \[F\mathbb{P}_\C X = F\bigvee_{n \geq 0}X^{\wedge n}/\Sigma_n \cong \bigvee_{n \geq 0}F(X^{\wedge n})/\Sigma_n \xrightarrow{\sim} \bigvee_{n \geq 0}(FX)^{\wedge n}/\Sigma_n = \mathbb{P}_\D FX\] is a weak equivalence. Hence using the isomorphism $\mathbb{P}_\D FX \cong \widetilde{F}\mathbb{P}_\C X$, if $X$ is cofibrant in $\C$, one sees that both $F\mathbb{P}_\C X \to \widetilde{F}\mathbb{P}_\C X$ and $F\mathbb{P}^n_\C X \to \widetilde{F}\mathbb{P}^n_\C X$ are weak equivalences.
	
We now prove that if
\[\begin{tikzcd}
\mathbb{P}_\C X \arrow[d] \arrow[r] & Y \arrow[d, "f"] \\
\mathbb{P}_\C X' \arrow[r] & P
\end{tikzcd} \]
is a pushout square in $\mathrm{CMon}(\C)$, and $FY \to \widetilde{F}Y$ is a weak equivalence where $Y$ is cofibrant, then $FP \to \widetilde{F}P$ is a weak equivalence. Since $I$ consists of cofibrations with cofibrant source, we may assume that $X$ and $X'$ are cofibrant in $\C$. By~\cite[B.2]{White17}, $f\colon Y \to P$ has a filtration $$Y = P_0 \to P_1 \to \cdots$$
where $P_{n-1} \to P_n$ is defined by the pushout
\[\begin{tikzcd}
Y \wedge Q_n(f)/\Sigma_n \arrow[d] \arrow[r] & P_{n-1} \arrow[d] \\
Y \wedge \mathbb{P}^n_\C X' \arrow[r] & P_n
\end{tikzcd} \]
in $\C$. For our purposes, it is not important precisely what $Q_n(f)$ is, apart from the fact that it is a colimit of a punctured $n$-dimensional cube whose vertices are given by tensor products of $X$, $X'$ and $Y$. It follows from the commutative monoid axiom that $Q_n(f)/\Sigma_n$ is cofibrant in $\C$, see~\cite[Proof of 4.17]{White17} for details.

Since $F$ sends pushouts in $\C$ to pushouts in $\D$, 
\[\begin{tikzcd}
F(Y \wedge Q_n(f)/\Sigma_n) \arrow[d] \arrow[r] & FP_{n-1} \arrow[d] \\
F(Y \wedge \mathbb{P}^n_\C X') \arrow[r] & FP_n
\end{tikzcd} \]
is a pushout in $\D$. Since $\widetilde{F}$ preserves pushouts of commutative monoids, 
\[\begin{tikzcd}
\mathbb{P}_\D FX \arrow[r] \arrow[d] & \widetilde{F}Y \arrow[d, "\widetilde{F}f"] \\
\mathbb{P}_\D FX' \arrow[r] & \widetilde{F}P
\end{tikzcd} \]
is a pushout in $\mathrm{CMon}(\D)$ using the isomorphism $\mathbb{P}_\D F \cong \widetilde{F}\mathbb{P}_\C.$ Applying~\cite[B.2]{White17} again, we obtain a filtration $\widetilde{F}Y = R_0 \to R_1 \to \cdots$ of $\widetilde{F}f\colon \widetilde{F}Y \to \widetilde{F}P$ where $R_{n-1} \to R_n$ is defined by the pushout
\[\begin{tikzcd}
\widetilde{F}Y \wedge Q_n(\widetilde{F}f)/\Sigma_n \arrow[d] \arrow[r] & R_{n-1} \arrow[d] \\
\widetilde{F}Y \wedge \mathbb{P}^n_\D FX' \arrow[r] & R_n
\end{tikzcd} \]
in $\D$. This filtration is compatible with the filtration of $Y \to P$ and therefore $\lambda_Y$ sends $FP_n$ to $R_n$. By applying $F$ to the pushout square shown Diagram 1 in~\cite[Proof of A.1]{White17} and using that $F \dashv G$ is a weak monoidal Quillen pair, one argues by induction that there is a natural weak equivalence $FQ_n(f) \xrightarrow{\sim} Q_n(\widetilde{F}f).$ Similarly, since taking orbits commutes with taking pushouts, there is a natural weak equivalence $FQ_n(f)/\Sigma_n \xrightarrow{\sim} Q_n(\widetilde{F}f)/\Sigma_n$ by an inductive argument on Diagram 6 in~\cite[Proof of A.3]{White17}.

We now show by induction that $\lambda_{P_n}\colon FP_n \to R_n$ is a weak equivalence. The base case holds since $\lambda_{P_0} = \lambda_Y$ which was a weak equivalence by assumption. Suppose that $\lambda_{P_{n-1}}$ is a weak equivalence. Consider the diagram
\[
\begin{tikzcd}
F(Y \wedge Q_n(f)/\Sigma_n) \arrow[dd] \arrow[rrr] \arrow[dr] & & & \widetilde{F}Y \wedge Q_n(\widetilde{F}f)/\Sigma_n \arrow[dd] \arrow[dr] & \\
& FP_{n-1} \arrow[rrr, crossing over] & & & R_{n-1} \arrow[dd] \\
F(Y \wedge \mathbb{P}^n_\C X') \arrow[dr] \arrow[rrr, crossing over] & & & \widetilde{F}Y \wedge \mathbb{P}^n_\D FX' \arrow[dr]  & \\
& FP_n \arrow[rrr] \arrow[uu, leftarrow, crossing over] & & & R_n 
\end{tikzcd}
 \]
in which the leftmost face and the rightmost face are pushouts in $\D$. The horizontal map $FP_{n-1} \to R_{n-1}$ is a weak equivalence by the inductive hypothesis. 

The horizontal map $F(Y \wedge Q_n(f)/\Sigma_n) \to \widetilde{F}Y \wedge Q_n(\widetilde{F}f)/\Sigma_n$ factors as the composite $$F(Y \wedge Q_n(f)/\Sigma_n) \to FY \wedge FQ_n(f)/\Sigma_n \to \widetilde{F}Y \wedge FQ_n(f)/\Sigma_n \to \widetilde{F}Y \wedge Q_n(\widetilde{F}f)/\Sigma_n$$ where the first map is a weak equivalence since $F \dashv G$ is a weak monoidal Quillen pair. The map $FY \to \widetilde{F}Y$ is a weak equivalence between cofibrant objects in $\widetilde{\D}$, and the map $FQ_n(f)/\Sigma_n \xrightarrow{\sim} Q_n(\widetilde{F}f)/\Sigma_n$ is a weak equivalence between cofibrant objects in $\D$. Since cofibrant objects in $\D$ are also cofibrant in $\widetilde{\D}$, both of these maps are weak equivalences between cofibrant objects in $\widetilde{\D}$. By Ken Brown's lemma, tensoring with cofibrant objects preserves weak equivalences between cofibrant objects, and hence the second and third map are weak equivalences. 

The horizontal map $F(Y \wedge \mathbb{P}^n_\C X') \to \widetilde{F}Y \wedge \mathbb{P}^n_\D FX'$ is a weak equivalence since it factors as the composite $$F(Y \wedge \mathbb{P}^n_\C X') \to FY \wedge F\mathbb{P}^n_\C X' \to \widetilde{F}Y \wedge \widetilde{F}\mathbb{P}^n_\C X' \cong \widetilde{F}Y \wedge \mathbb{P}^n_\D FX'.$$ Therefore, the map $FP_n \to R_n$ is a weak equivalence by~\cite[5.2.6]{Hovey99}. Each filtration map is a cofibration between cofibrant objects and hence by Ken Brown's lemma and~\cite[5.1.5]{Hovey99}, the map $FP \to \widetilde{F}P$ is a weak equivalence. 

It remains to show that the property is preserved under the transfinite compositions used to build relative cell complexes which again follows from~\cite[5.1.5]{Hovey99}. Therefore for any cofibrant commutative monoid object $A$ of $\C$, we have that the map $FA \to \widetilde{F}A$ is a weak equivalence which concludes the proof.
\end{proof}
\begin{rem}
	The hypothesis that $\C$ and $\D$ satisfy the commutative monoid axiom and the monoid axiom ensures that the categories of commutative monoids inherit a right lifted model structure~\cite[3.2]{White17}.
\end{rem}

\begin{rem}
In some cases such as rational chain complexes, cofibrant commutative algebras are cofibrant as modules, see Example~\ref{ex:convenient}. In such examples, one can take $\C = \widetilde{\C}$ in the previous theorem. 
\end{rem}

\begin{thm}\label{thm:commHQalg}
	There is a zig-zag of Quillen equivalences between the category of commutative $H\mathbb{Q}$-algebras and the category of commutative rational DGAs.
\end{thm}
\begin{proof}
Consider the adjunctions
	\[
\begin{tikzcd}
\mod{H\mathbb{Q}}^\mathrm{pf} \arrow[r, yshift=1mm, "Z"] & \symsp(\sq)_{\mathrm{pf}} \arrow[r, yshift=-1mm, "\phi^* N"'] \arrow[l,yshift=-1mm, "U"] & \symsp(\chq)_{\mathrm{pf}} \arrow[l, yshift=1mm, "L"'] \arrow[r, yshift=1mm, "D"] & \mathrm{Ch}_\mathbb{Q} \arrow[l, yshift=-1mm, "R"]
\end{tikzcd}\]
where $\mathrm{pf}$ denotes the positive flat model structure. Recall from Corollary~\ref{cor:stable and flat} that on $\symsp(\sq)$ and $\symsp(\chq)$ the positive stable and positive flat model structures are the same.

Firstly, we must justify that these are Quillen adjunctions. The adjunction $D \dashv R$ is Quillen since it can be viewed as the composite of Quillen adjunctions
\[\begin{tikzcd}
\symsp(\chq)_{\mathrm{pf}} \arrow[r, yshift=1mm, "1"] & \symsp(\chq) \arrow[l, yshift=-1mm, "1"] \arrow[r, yshift=1mm, "D"] & \mathrm{Ch}_\mathbb{Q} \arrow[l, yshift=-1mm, "R"]
\end{tikzcd} \]
where the second adjunction was proved to be Quillen in~\cite[4.7]{Shipley07}.

For the adjunction $Z \dashv U$, by Proposition~\ref{prop:Dugger} it is sufficient to check that the right adjoint $U$ preserves acyclic positive flat fibrations and positive flat fibrations between positive flat fibrants. Recall that a map is an acyclic positive flat fibration if and only if it is a levelwise acyclic blended fibration for levels $n > 0$, and that a map is a positive flat fibration between positive flat fibrants if and only if it is a levelwise blended fibration for levels $n > 0$ between positively flat fibrant objects. By~\cite[3.2.1]{PavlovScholbach18}, an object $X$ of $\symsp(\C,K)$ is positively flat fibrant if and only if $X$ is levelwise blended fibrant for levels $n>0$ and $X_n \to \underline{\mathrm{Hom}}(K,X_{n+1})$ is a naive weak equivalence for all $n\geq 0$ where $\underline{\mathrm{Hom}}(K,-)$ is the right adjoint to $K \otimes -$. One notes that all the conditions that must be checked, except for the last condition, are all levelwise. Therefore, applying the arguments given in Lemma~\ref{lem:Udetects} and to levels $n>0$ verifies the necessary levelwise conditions. The remaining condition that $X_n \to \underline{\mathrm{Hom}}(K,X_{n+1})$ is a naive weak equivalence is unchanged between the flat model structure and the positive flat model structure. This condition was also verified in Lemma~\ref{lem:Udetects}. For the $L \dashv \res[\phi]N$ adjunction one can argue similarly, using~\cite[4.4]{Shipley07}.

We now apply Theorem~\ref{thm:lift to comm mon}. For each of the categories of symmetric spectra, we take $\C$ to be the version equipped with the positive flat model structure, and $\widetilde{\C}$ to be equipped with the flat model structure. For the category of chain complexes, we take $\C = \widetilde{\C}$. In each case, cofibrant commutative algebras forget to flat cofibrant modules by Corollary~\ref{cor:flatcofibrants}. Hypothesis~\ref{hyp:orbits} holds for the first and last adjunctions since they are strong symmetric monoidal Quillen equivalences and therefore they give Quillen equivalences on the commutative monoids by Theorem~\ref{thm:lift to comm mon}. 

For the $L \dashv \res[\phi]N$ adjunction, we argue that condition (ii) in Lemma~\ref{lem:orbits} holds. We show that for a finite group $G$, if $f\colon X \to Y$ is an underlying cofibration in $G\text{-}\symsp(\sq)_{\mathrm{pf}}$ then $f$ is a naive cofibration in $G\text{-}\symsp(\sq)_{\mathrm{pf}}$. A $G$-object $X$ in $\symsp(\sq)$ consists of $G \times \Sigma_n$-objects $X(n)$ in $\sq$ with $G \times \Sigma_n$-equivariant structure maps. Similarly, a map $\phi\colon X \to Y$ between objects in $G\text{-}\symsp(\sq)$ consists of a collection of $G \times \Sigma_n$-equivariant maps $\phi(n)\colon X(n) \to Y(n)$ making the evident diagrams commute. 

Write $U$ for the forgetful functor $G\text{-}\symsp(\sq) \to \symsp(\sq)$. Suppose that $f\colon X \to Y$ is an underlying cofibration in $G\text{-}\symsp(\sq)_{\mathrm{pf}}$, i.e., $Uf\colon UX \to UY$ is a positive flat cofibration, and that $p\colon A \to B$ is a naive acyclic fibration in $G\text{-}\symsp(\sq)_{\mathrm{pf}}$, i.e., $Up$ is an acyclic positive flat fibration. Therefore $Uf$ has the left lifting property with respect to $Up$. It remains to argue that the lift $\theta\colon UY \to UA$ can be made into an equivariant map $\phi\colon Y \to A$. The lift $\theta\colon UY \to UA$ is a collection $\theta(n)\colon Y(n) \to A(n)$ of $\Sigma_n$-equivariant maps. Since the maps are determined levelwise, one can apply the averaging method as in the proof of Proposition~\ref{prop:injective} to construct $G \times \Sigma_n$-equivariant maps $\phi(n)\colon Y(n) \to A(n)$ and it follows that $\phi$ is a map in $G\text{-}\symsp(\sq)$ which is also a lift. Therefore, $f$ is a naive cofibration in $G\text{-}\symsp(\sq)$. Hence by Theorem~\ref{thm:lift to comm mon}, the middle Quillen equivalence also lifts to the commutative monoids.
\end{proof}

\section{A symmetric monoidal equivalence for modules}\label{sec:shipley2}
In this section, we give a symmetric monoidal Quillen equivalence between the categories of modules over a commutative $H\mathbb{Q}$-algebra and a commutative DGA. We note that this result has been assumed without proof in the literature; for more details see the introduction. We firstly explain why this result is not an immediate corollary of the zig-zag of Quillen equivalences $\mod{H\mathbb{Q}} \simeq_Q \mathrm{Ch}_\mathbb{Q}$. 

Let $F: \C \rightleftarrows \D :G$ be a strong symmetric monoidal Quillen equivalence and suppose that the unit objects of $\C$ and $\D$ are cofibrant. If $S$ is a \emph{cofibrant} monoid in $\C$, Schwede-Shipley~\cite[3.12]{SchwedeShipley03} show that $F:\mod{S}(\C) \rightleftarrows \mod{FS}(\D):G$ is a Quillen equivalence. Now suppose that $S$ is a commutative monoid in $\C$, which is not cofibrant as a monoid. Since $S$ is commutative, the category $\mod{S}(\C)$ of modules is symmetric monoidal, with tensor product defined by the coequalizer of the two maps
\begin{center}
	\begin{tikzcd}
	M \otimes_S N = \mathrm{coeq}(M \otimes S \otimes N \arrow[r, yshift=1mm] \arrow[r, yshift=-1mm] & M \otimes N)
	\end{tikzcd}
\end{center} defined by the action of $S$ on $M$ and $N$.

However, a cofibrant replacement $q\colon cS \xrightarrow{\sim} S$ as a monoid will no longer be commutative, and hence the zig-zag of Quillen equivalences \begin{center}
	\begin{tikzcd}
	\mod{S}(\C) \arrow[rr, "{\res[q]}"', yshift=-1mm] & & \mod{cS}(\C) \arrow[ll, "- \otimes_{cS} S"', yshift=1mm] \arrow[rr, "F", yshift=1mm] & & \mod{FcS}(\D) \arrow[ll, "G", yshift=-1mm]
	\end{tikzcd}
\end{center} 
cannot be symmetric monoidal. We explain how to rectify this. 

Before we can prove the desired symmetric monoidal Quillen equivalence, we require an abstract lemma about lifting symmetric monoidal Quillen equivalences to the categories of modules. We note that this first statement is a counterpart to~\cite[3.12(2)]{SchwedeShipley03}. The proof is effectively the same.
\begin{lem}\label{lem:cofremoval}
	Let \begin{center}\begin{tikzcd}\C \arrow[r, yshift=1mm, "F"] &  \arrow[l, yshift=-1mm, "G"] \D\end{tikzcd}\end{center} be a strong symmetric monoidal Quillen equivalence and let $S$ be a commutative monoid in $\C$. Suppose that $\C$ and $\D$ satisfy the monoid axiom. If $F$ preserves all weak equivalences and Quillen invariance holds in $\C$ and $\D$, then \begin{center}\begin{tikzcd}\mod{S} \arrow[r, yshift=1mm, "F"] &  \arrow[l, yshift=-1mm, "G"] \mod{FS}\end{tikzcd}\end{center} is a strong symmetric monoidal Quillen equivalence.
\end{lem}
\begin{proof}
	Let $q\colon cS \to S$ be a cofibrant replacement of $S$ as a monoid in $\C$. As $F$ preserves all weak equivalences $Fq\colon FcS \to FS$ is a weak equivalence. Consider the diagram of left Quillen functors
	\begin{center}
		\begin{tikzcd}
		\mod{S} \arrow[d, "F"'] & & \mod{cS} \arrow[ll, "S \otimes_{cS} -"'] \arrow[d, "F"] \\
		\mod{FS} & & \mod{FcS} \arrow[ll, "FS \otimes_{FcS} -"]
		\end{tikzcd}
	\end{center}
	which is commutative since $F$ is strong monoidal. By~\cite[3.12(1)]{SchwedeShipley03} the right hand vertical is a Quillen equivalence, and by Quillen invariance the horizontals are Quillen equivalences. Hence by 2-out-of-3 the left vertical is a Quillen equivalence as required. As a functor between the module categories, $F$ is strong symmetric monoidal since the tensor product in the module category $\mod{S}$ is defined by a coequalizer which $F$ preserves. Therefore \begin{center}\begin{tikzcd}\mod{S} \arrow[r, yshift=1mm, "F"] &  \arrow[l, yshift=-1mm, "G"] \mod{FS}\end{tikzcd}\end{center} is a strong symmetric monoidal Quillen equivalence.
\end{proof}

We recall from Shipley~\cite[1.2]{Shipley07} the zig-zag of natural weak equivalences between $Zc$ and $\res[\alpha]\widetilde{\mathbb{Q}}$ where $\alpha$ is the ring map $H\mathbb{Q} \to \widetilde{\mathbb{Q}}H\mathbb{Q}$ induced by the unit of the monad structure on $\widetilde{\mathbb{Q}}$. Let $\beta\colon \widetilde{\mathbb{Q}}H\mathbb{Q} \to H\mathbb{Q}$ be the ring map induced by the multiplication map of the monad structure.

We have $Zc = \ext[\beta]\widetilde{\mathbb{Q}}c \cong \res[\alpha]\res[\beta]\ext[\beta]\widetilde{\mathbb{Q}}c$ since $\beta\alpha = 1$. There is then a natural map $\res[\alpha]\widetilde{\mathbb{Q}}c \to \res[\alpha]\res[\beta]\ext[\beta]\widetilde{\mathbb{Q}}c$ arising from the unit of the $\ext[\beta] \dashv \res[\beta]$ adjunction. This is a weak equivalence since $\widetilde{\mathbb{Q}}$ preserves cofibrant objects, the $\ext[\beta] \dashv \res[\beta]$ adjunction is a Quillen equivalence and $\res[\alpha]$ preserves all weak equivalences. Finally there is a natural map $\res[\alpha]\widetilde{\mathbb{Q}}c \to \res[\alpha]\widetilde{\mathbb{Q}}$ which is a weak equivalence as $\res[\alpha]$ and $\widetilde{\mathbb{Q}}$ preserve all weak equivalences. We can now apply the previous lemma to obtain the desired statement. 
\begin{thm}\label{thm:symmmon}
	Let $A$ be a commutative $H\mathbb{Q}$-algebra. There are zig-zags of weak symmetric monoidal Quillen equivalences \[\mod{A}^\mathrm{stable} \simeq_Q \mod{\underline{\Theta} A} \quad \text{and} \quad \mod{A}^\mathrm{flat} \simeq_Q \mod{\underline{\Theta} A}\] where $\underline{\Theta} A = D\res[\phi]N\res[\alpha]\widetilde{\mathbb{Q}}A$ is a commutative DGA.
\end{thm}
\begin{proof}
The proof for each part of the theorem follows the same method. Namely, we apply~\cite[3.12(2)]{SchwedeShipley03} together with Lemma~\ref{lem:cofremoval} to the underlying Quillen equivalences given by Shipley~\cite{Shipley07} in the stable case, and given by Theorem~\ref{thm:main1} in the flat case. Since the weak equivalences in both the stable model structure and the flat model structures are the same, the following proof applies in both cases.

The first step is the adjunction
\[\begin{tikzcd}\mod{A}(\mod{H\mathbb{Q}}) \arrow[r, yshift=1mm, "\widetilde{\mathbb{Q}}"] & \mod{\widetilde{\mathbb{Q}}A}\left(\mod{\widetilde{\mathbb{Q}}H\mathbb{Q}}\right). \arrow[l, yshift=-1mm, "U"]\end{tikzcd}\] Since $\widetilde{\mathbb{Q}}$ preserves all weak equivalences, this is a strong symmetric monoidal Quillen adjunction by Lemma~\ref{lem:cofremoval}. 

Recall that there is a ring map $\alpha\colon H\mathbb{Q} \to \widetilde{\mathbb{Q}}H\mathbb{Q}$. Since $\res[\alpha]$ is lax symmetric monoidal it gives rise to a functor
$$\mod{\widetilde{\mathbb{Q}}A}\left(\mod{\widetilde{\mathbb{Q}}H\mathbb{Q}}\right) \xrightarrow{\res[\alpha]} \mod{\res[\alpha]\widetilde{\mathbb{Q}}A}(\Sp^\Sigma(\sq)).$$
It follows from~\cite[\S 3.3]{SchwedeShipley03} that the left adjoint to $\res[\alpha]$ at the level of modules, is given by $\ext[\alpha]^{\widetilde{\mathbb{Q}}A}(M) = \widetilde{\mathbb{Q}}A \otimes_{\res[\alpha]\ext[\alpha]\widetilde{\mathbb{Q}}A} \ext[\alpha]M$. We claim that $\ext[\alpha]^{\widetilde{\mathbb{Q}}A}$ is strong monoidal. As $\ext[\alpha]$ preserves colimits and is strong monoidal, we have
\begin{align*} 
\ext[\alpha](M \otimes_{\res[\alpha]\widetilde{\mathbb{Q}}A} N) &= \ext[\alpha]\mathrm{coeq}(M \otimes \res[\alpha]\widetilde{\mathbb{Q}}A \otimes N \rightrightarrows M \otimes N) \\
&\cong \mathrm{coeq}(\ext[\alpha]M \otimes \ext[\alpha]\res[\alpha]\widetilde{\mathbb{Q}}A \otimes \ext[\alpha]N \rightrightarrows \ext[\alpha]M \otimes \ext[\alpha]N) \\
&= \ext[\alpha]M \otimes_{\ext[\alpha]\res[\alpha]\widetilde{\mathbb{Q}}A} \ext[\alpha]N. \end{align*}
From this, one sees that $$\ext[\alpha]^{\widetilde{\mathbb{Q}}A}(M \otimes_{\res[\alpha]\widetilde{\mathbb{Q}}A} N) \cong \ext[\alpha]^{\widetilde{\mathbb{Q}}A}(M) \otimes_{\widetilde{\mathbb{Q}}A} \ext[\alpha]^{\widetilde{\mathbb{Q}}A}(N)$$ and hence $\ext[\alpha]^{\widetilde{\mathbb{Q}}A}$ is strong symmetric monoidal. Since $\res[\alpha]$ preserves all weak equivalences, it follows from~\cite[3.12(2)]{SchwedeShipley03} that
\[\begin{tikzcd}
\mod{\widetilde{\mathbb{Q}}A}\left(\mod{\widetilde{\mathbb{Q}}H\mathbb{Q}}\right) \arrow[r, "{\res[\alpha]}"', yshift=-1mm] & \mod{\res[\alpha]\widetilde{\mathbb{Q}}A}(\Sp^\Sigma(\sq)) \arrow[l, "{\ext[\alpha]^{\widetilde{\mathbb{Q}}A}}"', yshift=1mm]
\end{tikzcd}\]
is a strong symmetric monoidal Quillen equivalence.

	The next step is the passage along the Dold-Kan type equivalence. 
	Recall that applying the normalization functor levelwise gives a lax monoidal functor $$\Sp^\Sigma(\sq) \to \mod{\mc{N}}\left((\chq)^\Sigma \right)$$ where $\mc{N} = N(\mathrm{Sym}(\widetilde{\mathbb{Q}}S^1)$, and that there is a ring map $\phi\colon \mathrm{Sym}(\mathbb{Q}[1]) \to \mc{N}$. The composite $\res[\phi]N\colon \Sp^\Sigma(\sq) \to \Sp^\Sigma(\chq)$ is lax monoidal.
	
	Let $S$ be a commutative monoid in $\Sp^\Sigma(\sq)$.  We now show that the induced functor $$\res[\phi]N\colon \mod{S}(\Sp^\Sigma(\sq)) \to \mod{\res[\phi]NS}(\Sp^\Sigma(\chq))$$ on the categories of modules is lax symmetric monoidal. 
	 Recall that colimits in categories of modules are calculated in the underlying category of symmetric spectra where they are computed levelwise. Therefore, $N$ preserves colimits as it is an equivalence of categories $\sq \to \chq$. The restriction of scalars $\res[\phi]$ also preserves colimits since it is left adjoint to the coextension of scalars functor. Therefore we have the following map
	\begin{align*}
	\res[\phi]NA \otimes_{\res[\phi]NS} \res[\phi]NB &= \mathrm{coeq}(\res[\phi]NA \otimes \res[\phi]NS \otimes \res[\phi]NB \rightrightarrows \res[\phi]NA \otimes \res[\phi]NB) \\
	&\to \mathrm{coeq}(\res[\phi]N(A \otimes S \otimes B) \rightrightarrows \res[\phi]N(A \otimes B)) \\
	&\cong \res[\phi]N(\mathrm{coeq}(A \otimes S \otimes B \rightrightarrows A \otimes B) \\
	&= \res[\phi]N(A \otimes_S B)
	\end{align*}
giving $\res[\phi]N$ a lax symmetric monoidal structure as a functor between the categories of modules. We now must show that $L^{S} \dashv \res[\phi]N$ is a weak monoidal Quillen pair, where $L^S$ denotes the left adjoint of $\res[\phi]N$. We use the criteria~\cite[3.17]{SchwedeShipley03}. Since the monoidal unit $\res[\phi]NS$ is cofibrant in $\mod{\res[\phi]NS}$ the first condition is that $L^S\res[\phi]NS \to S$ is a weak equivalence. Since $\res[\phi]N$ preserves all weak equivalences~\cite[4.4]{Shipley07} and $\res[\phi]NS$ is cofibrant, this map is the derived counit of the Quillen equivalence $L^S \dashv \res[\phi]N$ and as such is a weak equivalence. The second condition holds since $\res[\phi]NS$ is a generator for the homotopy category of $\mod{\res[\phi]NS}.$ 

By taking $S = \res[\alpha]\widetilde{\mathbb{Q}}A$ in the previous discussion, the adjunction 
\[\begin{tikzcd}\mod{\res[\alpha]\widetilde{\mathbb{Q}}A}(\Sp^\Sigma(\sq)) \arrow[rr, yshift=-1mm, "{\res[\phi]N}"'] & & \mod{\res[\phi]N\res[\alpha]\widetilde{\mathbb{Q}}A}(\Sp^\Sigma(\chq)) \arrow[ll, yshift=1mm, "L^{\res[\alpha]\widetilde{\mathbb{Q}}A}"'] \end{tikzcd}\]
is a weak symmetric monoidal Quillen adjunction. Since $\res[\phi]N$ preserves all weak equivalences~\cite[4.4]{Shipley07}, it follows from~\cite[3.12(2)]{SchwedeShipley03} that this is moreover a weak symmetric monoidal Quillen equivalence.

The final step in the zig-zag is the adjunction 
\[\begin{tikzcd}
\mod{\res[\phi]N\res[\alpha]\widetilde{\mathbb{Q}}A}(\Sp^\Sigma(\chq)) \arrow[r, yshift=1mm, "D"] & \mod{\underline{\Theta}A}(\mathrm{Ch}_\mathbb{Q}) \arrow[l, yshift=-1mm, "R"]
 \end{tikzcd}\]
 which is a strong symmetric monoidal Quillen equivalence by Lemma~\ref{lem:cofremoval}, since $D$ preserves all weak equivalences rationally~\cite[4.8]{Shipley07}.
\end{proof}

\bibliographystyle{plain}
\bibliography{../mybib}
\end{document}